\documentclass[11pt]{article}

\usepackage{amsfonts,amssymb,amsthm,amsmath,comment}
\usepackage{hyperref}

\usepackage{algorithm}
\usepackage{algpseudocode}

\usepackage{todonotes}

\usepackage{bm}

\usepackage{fancybox}
\usepackage{color}
\usepackage{latexsym}

\usepackage{eepic}
\usepackage{epic}
\usepackage{epsf}
\usepackage{graphics}

\newtheorem{theorem}{Theorem}

\newtheorem{lemma}[theorem]{Lemma}

\newtheorem{claim}[theorem]{Claim}

\newtheorem{proposition}[theorem]{Proposition}
\theoremstyle{definition}

\newtheorem{definition}[theorem]{Definition}
\newtheorem{remark}[theorem]{Remark}

\newenvironment{claimproof}[1]{\par\noindent{\em Proof of Claim.}~#1}

\newcommand{\F}{\mathbb{F}}

\newcommand{\M}{\mathcal{M}}

\renewcommand{\ge}{\geqslant}

\renewcommand{\le}{\leqslant}

\renewcommand{\angle}[1]{\langle #1 \rangle}

\newcommand{\sympl}[1]{\langle #1 \rangle}

\title{\textbf{Recognizing the Commuting Graph of a Finite
    Group}\thanks{This work was initiated during the online series of
    Research Seminars on ``Groups and Graphs'', March-August, 2021,
    run by Ambat Vijaykumar and Aparna Lakshmanan, Cochin Univ of
    Science and Technology.}}

\date{}

\author{V. Arvind\thanks{Institute of Mathematical Sciences (HBNI),
    Chennai, India. {\tt arvind@imsc.res.in}} \and Peter Cameron
  \thanks{Mathematics Department, University of St. Andrews, UK.  {\tt
      pjc20@st-andrews.ac.uk}.}}

\begin{document}
\maketitle
\thispagestyle{empty}

\begin{abstract}
In this paper we study the realizability question for commuting graphs
of finite groups: Given an undirected graph $X$ is it the commuting
graph of a group $G$? And if so, to determine such a group. We seek
efficient algorithms for this problem. We make some general
observations on this problem, and obtain a polynomial-time algorithm
for the case of extraspecial groups.
\end{abstract}

\section{Introduction}

The commuting graph $\Gamma(G)$ of a finite group $G$ is a simple
undirected graph with vertex set $G$ and undirected edges $(x,y)$ for
each pair of commuting elements $x\ne y\in G$. There are variations of
this definition in the literature. For example, often the center
$Z(G)$ is removed from the graph, because vertices in $Z(G)$ are
adjacent to every vertex. 

The commuting graph of a group has been a topic of research for over
sixty years with a variety of results about properties of the
commuting graph. The earliest reference to commuting graphs, arguably,
is the seminal paper of Brauer and Fowler on centralizers of
involutions in simple groups \cite{BF55}, where it is used though not
explicitly defined.

Our focus is on the recognition problem of commuting graphs. It is an
algorithmic problem: given an undirected graph $X=(V,E)$ as input, we
want to check if there is a group $G$ with $|V|$ elements such that
$X$ is isomorphic to $\Gamma(G)$. Our main results are:

\begin{itemize}
\item A deterministic polynomial time algorithm for the case of
  extraspecial $p$-groups (which are a special case of $p$-groups of
  nilpotence class $2$). 
\item A quasipolynomial time algorithm in the general case, based on
  short (i.e., $O(\log^3n)$ size) presentations for finite groups of
  order $n$ combined with Babai's quasipolynomial time algorithm for
  graph isomorphism.  
\end{itemize}

A natural question in connection with our algorithm for recognizing
the commuting graphs of extraspecial groups is whether groups with the
same commuting graphs are isoclinic. This holds for extraspecial
groups which is exploited by the algorithm, and it is natural to
conjecture that this property holds for all groups of nilpotence class
2 (of which extraspecial groups are a subclass). We present
counter-examples for class-3 nilpotent groups and conjecture that the
property holds for class-2 nilpotent groups.

Additionally, we have some other observations: an efficient reduction
of the problem to recognizing the commuting graphs of indecomposable
groups, recognizing the commuting graph of dihedral groups along with
a generalization to Frobenius groups.

\noindent\textbf{Some related work.}~ There is a result by Giudici and
Kuzma \cite{GK16} that shows the following: every $n$-vertex graph $X$
with at least two vertices of degree $n-1$ is realizable as the
commuting graph of a semigroup. It is easy to see that their
construction actually gives a polynomial-time algorithm for finding a
semigroup with $n$ elements and a bijection from it to the vertex set
of $X$ such that the edges of $X$ realize the commuting relation of
the semigroup. It is a nearly complete answer to the question in the
semigroups setting.

 Solomon and Woldar \cite{SW13} have shown that the commuting graph
 $\Gamma(G)$ of a finite simple group $G$ is uniquely determined by
 the group. That is, $\Gamma(H)\simeq \Gamma(G)$ if and only if
 $G\simeq H$. We believe that checking if $X$ is the commuting graph
 of a simple group should be possible in polynomial time.

\section{Basic properties}

We being with some preliminary observations that are well-known in the
literature (see, e.g., the survey \cite{GGsurvey}).

Let $G$ be a finite group. What are the cliques of $\Gamma(G)$? If a
vertex subset $S$ induces a clique in $\Gamma(G)$ then $S$ is a
commuting subset of elements of $G$. Conversely, every commuting
subset of elements of $G$ forms a clique in $\Gamma(G)$. Which cliques
of $\Gamma(G)$ correspond to subgroups of $G$? Although we cannot
directly infer the group multiplication from $\Gamma(G)$, we can
observe that

\begin{lemma}
  A vertex subset $S$ is a \emph{maximal} clique of $X$ iff $S$ is a
  maximal abelian subgroup of $G$.
\end{lemma}

\begin{proof}
  Suppose $H$ is a maximal abelian subgroup of $G$. Clearly, $H$ is a
  clique in $X$. If $x\notin H$ is adjacent to all of $H$ then $x$
  commutes with all of $H$ implying that $\angle{H\cup\{x\}}$ is an
  abelian subgroup of $G$ strictly larger than $H$. Hence the clique
  induced by $H$ is maximal.  Conversely, by a similar argument, if
  $S$ is a maximal clique in $X$ then $S$ is a maximal abelian
  subgroup of $G$.
\end{proof}

\subsection*{The vertex degrees of $\Gamma(G)$ and conjugacy classes of $G$}

Let $X=(G,E)$ be the commuting graph of a finite group $G$.  Let $x^G$
denote the \emph{conjugacy class} of $x\in G$:
\[
x^G = \{g^{-1}xg\mid g\in G\},
\]
which is the orbit of $x$ under the conjugation action of $G$ on
itself. Let $\deg(x)$ denote the degree of a node $x$ in the graph
$X$. The closed neighborhood $\bar{N}(v)$ of any vertex $v$ of the
commuting graph $\Gamma(G)$ is defined as
\[
\bar{N}(x) = \{u\in G\mid u=v \text{ or } (u,v)\in E\}.
\]
The orbit-stabilizer lemma \cite{permgps} directly implies the
following

\begin{proposition}
  For each $x\in G$ its centralizer $C_G(x)$ is the closed neighborhood
  $\bar{N}(x)$ of $x$ in $\Gamma(G)$, and
  \[
|C_G(x)| = 1+ \deg(x) = {\frac{|G|}{|x^G|}}, \textrm{ for all }x\in G.
  \]
\end{proposition}

Let $m$ denote the number of edges in the commuting graph $\Gamma(G)$,
and let $k$ denote the number of conjugacy classes.  As $\sum_{x\in
  G}\deg(x) = 2m$, we have:

\begin{equation}
2m +n = \sum_{x\in G} {\frac{|G|}{|x^G|}} = |G|\cdot \sum_{x\in
  G}{\frac{1}{|x^G|}} = n\cdot k.
\end{equation}

Thus the number of conjugacy classes of $G$ is $k=(2m+n)/n$, which, by
the above equation, can be inferred from the commuting graph. Thus,
for example, the only regular commuting graphs are the complete
graphs, which are the commuting graphs of abelian groups.

The problem of the minimum number $f(n)$ of conjugacy classes in a group
of order~$n$ has a long history. Landau~\cite{La} showed that $f(n)\to\infty$
as $n\to\infty$. The first lower bound was by Brauer~\cite{Brauer} and
Erd\H{o}s and Tur\'an~\cite{ET}, who showed that $f(n)\ge\log\log n$
(logarithms to base~$2$). This was improved to $\epsilon\log n/(\log\log n)^8$
by Laci Pyber~\cite{Py}. The exponent $8$ was reduced to $7$ by Thomas Keller
\cite{Keller}, and to $3+\epsilon$ by Barbara Baumeister, Attila Mar\'oti and
Hung Tong-Viet~\cite{BMT-V}. It is conjectured that a bound of the form
$f(n)\ge C\log n$ holds for some constant $C$. In the other direction, 
$f(n)\le(\log n)^3$.

This is relevant to us because an $n$-vertex graph $X$ with $o(f(n))$
edges cannot be the commuting graph of an $n$ element group.

At the other extreme, we can rule out very dense incomplete graphs by
the $5/8$-theorem \cite{Gu73} for finite groups: any graph $X$ that is
not complete and has more than $5/8\cdot n^2$ edges cannot be the
commuting graph of an $n$-element group.

\subsection*{The commuting graph and maximal abelian subgroups}

For a finite group $G$, let $\M\subseteq 2^G$ denote the set of all
\emph{maximal abelian subgroups} of $G$. Associated with $G$ is the
natural \emph{hypergraph} $(G,\M)$, where the hyperedges are precisely
the maximal abelian subgroups of $G$.

\begin{proposition}\label{max-abel-subgps}
  The commuting graph of a finite group determines the hypergraph of
  its maximal abelian subgroups, and, conversely the maximal abelian
  subgroups hypergraph of the group determines its commuting graph.
\end{proposition}

\begin{proof}
  Let $G$ be a finite group. Clearly, from the commuting graph
  $\Gamma(G)$ we can determine all the maximal cliques which
  corresponds to all maximal abelian subgroups of $G$ which implies
  that the hypergraph of maximal abelian subgroups is determined by
  $\Gamma(G)$. Conversely, given the hypergraph $(G,\M)$ we define the
  edge set $E=\{\{u,v\}\mid u,v\in A\text{ for some }A\in\M\}$.
Clearly, $E$ is the edge set of the commuting graph $\Gamma(G)$.
\end{proof}

\begin{remark}
Since the commuting graph $\Gamma(G)$ of a finite simple group $G$ is
uniquely determined \cite{SW13}, by Proposition~\ref{max-abel-subgps}
it follows that the set of maximal abelian groups of a finite simple
group $G$ uniquely determines $G$.

For a finite group $G$, the number of maximal abelian subgroups is
bounded by ${|G|}\choose {\log|G|}$ because every subgroup of $G$ has
a generating set of size bounded by $\log|G|$. Thus, the hypergraph of
maximal abelian subgroups has size at most $n^{\log n}$ for $n$
element groups. 

This simple bound is tight apart from a constant in the exponent. To see this,
consider the extraspecial group $G$ of order $p^{2n+1}$ and exponent $p$,
where $p$ is an odd prime. The centre has order $p$, and $G/Z(G)$ is isomorphic
to a $2n$-dimensional vector space $V$ over the field $F$ of $p$ elements, with
the bilinear form from $V\times V$ to $F$ corresponding to the commutation map
from $G/Z(G)\times G/Z(G)$ to $Z(G)$. Maximal abelian subgroups contain the
centre, and correspond to maximal totally isotropic subspaces of $V$. It is
known that the number of such subspaces is $\prod_{i=1}^n(p^i+1)$ (see
\cite{Tay}), which is greater than $p^{n(n+1)}/2$, roughly $|G|^{n/4}$.
\end{remark}

\section{Commuting graphs of product groups}

Let $G$ and $H$ be finite groups. We now consider the commuting graph
$\Gamma(G\times H)$ of their direct product.

Let $X=(V,E)$ and $X'=(V',E')$ be simple undirected graphs. Recall
\cite{IK08book} that the \emph{strong product} of the graphs $X$ and
$X'$, denoted $X\boxtimes X'$ is a simple undirected graph with the
cartesian product $V\times V'$ as its vertex set and edges defined as
follows: distinct pairs $(u,u')$ and $(v,v')$ are adjacent if and only
if one of the following holds:
\begin{itemize}
\item $u=u'$ and $(v,v')\in E'$,
\item $v=v'$ and $(u,u')\in E$,
\item $(u,u')\in E$ and $(v,v')\in E'$.
\end{itemize}

The following proposition is immediate from the definition.

\begin{proposition}
  For finite groups $G$ and $H$
  \[
  \Gamma(G\times H) = \Gamma(G)\boxtimes \Gamma(H).
  \]
\end{proposition}

Since simple undirected graphs can be uniquely factorized into strong
products of prime graphs \cite{IK08book}, which can be computed in
polynomial time \cite{FeSch92}, we can derive the following
reduction. Recall the a group $G$ is said to be \emph{indecomposable}
if it is not the direct product of two non-trivial groups.

\begin{theorem}\label{product-thm}
  The problem of recognizing the commuting graphs of groups is
  polynomial-time reducible to the problem of recognizing the
  commuting graphs of indecomposable groups.
\end{theorem}

\begin{proof}
  Suppose we have an algorithm $\mathcal{A}$ for recognizing the
  commuting graphs of indecomposable groups. Using $\mathcal{A}$ as
  subroutine, we present a polynomial-time algorithm for recognizing
  the commuting graphs of all finite groups.
  
  Let $X=(V,E)$ be an undirected graph on $n$ vertices which is a
  purported commuting graph.

  First, using the polynomial-time algorithm of Feigenbaum and
  Sch\"affer we can factorize $X$ as
  \[
  X=X_1\boxtimes X_2\boxtimes\cdots \boxtimes X_k,
  \]
  where each $X_i$ is a prime graph on at least two vertices.  It
  follows that $k\le\log n$. Now, for each subset $S\subseteq [k]$
  of the prime graphs we combine them by taking the strong
  direct product to define the graph
  \[
  X_S =\boxtimes_{i\in S} X_i.
  \]
  Notice that any order in which the strong product of these graphs
  $X_i$ is computed yields the same graph, up to isomorphism.

  Thus, we have computed graphs $X_S$ for each subset $S$ of $[k]$.
  Now, we invoke the subroutine $\mathcal{A}$ that check if $X_S$ is
  the commuting graph of an indecomposable group $G_S$, and if so,
  finds a labeling of the vertices of $X_S$ with elements of $G_S$
  consistent with the commuting relations.

  We can now check if the input graph $X$ is the commuting graph of a
  group with a straightforward dynamic programming strategy based on
  the following easy claim.

\begin{claim}  
  For any two disjoint subsets $S,S'$ of $[k]$ such that $X_S$ and
  $X_{S'}$ are the commuting graphs of groups $G_S$ and $G_{S'}$ the
  graph $X_S\boxtimes X_{S'}$ is the commuting graph of the direct
  product group $G_S\times G_{S'}$.
\end{claim} 

Now, the algorithm works in stages, computing subsets $S$ of $[k]$
along with the graph $X_S$ and group $G_S$ such that
$X_S=\Gamma(G_S)$. 
\begin{enumerate}
\item \textbf{for} stages $0$ \textbf{to} $k$ \textbf{do}
\item \textbf{Stage $0$}~~ we have subsets $S$ such that $G_S$ is an
  indecomposable group. We mark all such subsets $S$ as true. We mark
  the remaining subsets as false.
\item \textbf{Stage $i+1$}~~ For each pair of disjoint subsets $S$ and
  $S'$ marked true in Stages $1,2,\ldots,i$, such that $S\cup S'$ is
  marked false, we mark $S\cup S'$ true and compute $X_{S\cup
    S'}=X_S\boxtimes X_{S'}$ and $G_{S\cup S'}=G_S\times G_{S'}$.
\item \textbf{end-for}
\item If $[k]$ is marked true then the input $X$ is the commuting
  graph of the group $G_{[k]}$ computed above.
\end{enumerate}

The above description checks if $X$ is the commuting graph of a group
with at most $2^k\le n$ calls to the subroutine $\mathcal{A}$ and the
running time of the remaining computation is clearly polynomially
bounded in $n$.
\end{proof}

\begin{remark}
 From Theorem~\ref{product-thm} we can easily deduce that the Solomon
 Woldar theorem \cite{SW13} implies that the direct product of simple
 groups too have uniquely determined commuting graphs. In fact, the
algorithm can be simplified in this case; the following result shows that
we only need to consider the indecomposable factors, not arbitrary sums of them.
\end{remark}  

\begin{proposition}
The commuting graph of a finite simple group is a prime graph under the 
strong product.
\end{proposition}

\begin{proof}
We use the fact that, if $G$ is a non-abelian simple group and
$g\in G\setminus\{1\}$, then there exists $h\in G$ such that
$\langle g,h\rangle=G$~\cite{GK}. Now if $\langle g,h\rangle=G$, then
\begin{itemize}
\item $g$ and $h$ are nonadjacent in the commuting graph (since $G$
is non-abelian);
\item $g$ and $h$ have no non-identity common neighbour in the commuting
graph (since a common neighbour would belong to $Z(G)$, but $Z(G)=\{1\}$.
\end{itemize}

Now suppose for a contradiction that $\Gamma(G)$ is the strong product
of two nontrivial graphs with vertex sets $A$ and $B$ (i.e., with
$|A|>1$ and $|B|>1$). Then we can identify $G$ with the Cartesian
product $A\times B$. Suppose that $(a,b)$ is the identity of $G$. Then
$(a,b)$ is joined to all other vertices in the commuting graph.

It follows that $a$ is joined to every other vertex in $A$, and $b$ to
every other vertex in $B$; therefore, for all $x\in A\setminus\{a\}$,
$y\in B\setminus\{b\}$, the three vertices $(a,y)$, $(x,b)$ and
$(x,y)$ are adjacent to each other in the commuting graph $\Gamma(G)$.


Choose $y\in B\setminus\{b\}$, and suppose that
$\langle(a,y),(u,v)\rangle=G$. Now
\begin{itemize}
\item if $u=a$ then $(a,y)$ and $(a,v)$ are both joined to $(x,b)$
for any $x\in A\setminus\{a\}$;
\item if $v=b$ then $(a,y)$ and $(u,b)$ are joined;
\item if neither of the above, then $(a,y)$ and $(u,v)$ are both
joined to $(u,b)$.
\end{itemize}
Each case is contradictory; so our assumption that $\Gamma(G)$ is the
strong product of two nontrivial graphs is false, and the theorem is
proved.
\end{proof}

The proof only requires that $Z(G)=1$ and that any non-identity
element is contained in a $2$-element generating set. These
assumptions are valid in any almost simple group $G$ with simple
normal subgroup $S$ such that $G/S$ is cyclic~\cite{BGH}.

\section{Commuting graphs of semidirect products}

In this section we explore whether we can recognize the commuting
graph of semidirect products $G=H\rtimes K$ if $H$ and $K$ are both
from group classes whose commuting graphs are easily recognizable.

For example, consider the commuting graph of the dihedral group $D_n$.
Let $D_n=\angle{a,b}$ where $a^2=1, b^n=1$ and $aba^{-1}=b^{-1}$. For
$n$ odd, $1$ is the only dominant vertex, there is an $n$-clique
corresponding to $\angle{b}$ and the $ab^i$ are pendant vertices. For
$n$ even, it is a bit different with $b^{n/2}$ being the other
dominant vertex.


Seeking a generalization of this example we first consider Frobenius
groups.

\subsection*{Commuting graphs of Frobenius groups}\label{frob-sec}

In this subsection we demonstrate that we can recognise from its
commuting graph that a group $G$ is a Frobenius group.

A finite group $G$ is a \emph{Frobenius group} if it contains a
non-trivial proper subgroup $H$, the \emph{Frobenius complement}, with
the property that $H\cap H^g=\{1\}$ for all $g\notin H$. The theorem
of Frobenius asserts that a Frobenius group has a normal subgroup $N$,
the \emph{Frobenius kernel}, such that $NH=G$ and $N\cap H=\{1\}$;
every non-identity element of $G$ is in either the Frobenius kernel or
a conjugate of the Frobenius complement.  Thompson proved that a
Frobenius kernel is nilpotent, and Zassenhaus worked out the detailed
structure of a Frobenius complement.

An alternative definition is that $G$ is a Frobenius group if it is isomorphic
to a transitive permutation group which is not regular but in which the 
stabilizer of any two points is the identity.

Everything we need about Frobenius groups is contained in Passman's
book~\cite{passman}.

In the commuting graph of a group $G$, the identity is a \emph{dominant}
vertex (that is, joined to all others); indeed, any vertex in the centre
is dominant, so if $Z(G)$ is non-trivial then the commuting graph is
$2$-connected.

\begin{lemma}
A Frobenius kernel has non-trivial centre.
\end{lemma}

\begin{proof} By Thompson's theorem, a Frobenius kernel is nilpotent.
\end{proof}

\begin{lemma}
A Frobenius complement has non-trivial centre.
\end{lemma}

\begin{proof} Suppose that $H$ is a Frobenius complement.

If $H$ has even order, then it contains a unique involution (which acts on
the Frobenius kernel as inversion -- so in this case the Frobenius kernel is
abelian). This involution is joined to all other vertices.

If $H$ has odd order, then we use the fact that any subgroup whose order is the
product of two primes is cyclic. So all the Sylow subgroups of $G$ are cyclic.
Suppose that the prime divisors of $|H|$ are $p_1,p_2,\ldots,p_r$ in order.

Now $H$ is metacyclic; its Fitting subgroup $F$ is cyclic and contains its
centralizer. If $F$ contains a subgroup $P$ of order $p_1$, then this subgroup
is normal in $G$; and conversely, a normal subgroup of prime order is contained
in $F$. Let $P$ be a subgroup of order $p_1$ and suppose that $P\not\le F$.
Then $P$ normalizes but does not centralize $F$, so $P$ must act non-trivially
on a cyclic $p$-subgroup of $F$, and hence on a cyclic subgroup of order 
$p\ne p_1$; then $G$ has a non-abelian subgroup of order $pp_1$, a
contradiction. So $F$ contains a cyclic subgroup of $P$ order $p_1$.

Now $P$ is normal in $G$, and so $P\le Z(G)$ as required, since its
automorphism group is divisible only by primes smaller than $P$.
\end{proof}

\begin{theorem}
Let $G$ be a group of order $nk$, where $n,k>1$ and $\gcd(n,k)=1$, and let
$\Gamma$ be the commuting graph of $G$. Then $G$ is a Frobenius group with
Frobenius complement of order $k$ if and only if $\Gamma$ satisfies the
following conditions:
\begin{itemize}
\item[(a)] there is a dominant vertex $v$;
\item[(b)] $\Gamma\setminus\{v\}$ has a component of size $n-1$ and $n$
components of size $k-1$, and each component has a dominant vertex.
\end{itemize}
\end{theorem}

\begin{proof}
Suppose that $G$ is a Frobenius group with kernel $N$ of order $n$ and 
complement $H$ of order $k$. The identity is a dominant vertex. Also
non-identity vertices in the kernel do not commute with non-identity vertices
in a complement, and non-identity vertices in different complements do not
commute with each other. So the components of $\Gamma\setminus\{1\}$ are
as stated, and the lemmas above show that they have dominant vertices.

Conversely, suppose that the commuting graph $\Gamma$ of $G$ has properties
(a) and (b). If $C$ is a component with a dominant vertex $c$, then the
centralizer $C_G(c)$ is equal to $C\cup\{1\}$, which is thus a subgroup of $G$.
Let $N$ be the subgroup containing the component of size $n-1$, and let
$H_1,\ldots,H_n$ be the subgroups containing the other components.

Since $\Gamma$ is invariant under automorphisms of $G$, we have an action of
$G$ on the set $\Omega$ of components of $\Gamma\setminus\{v\}$ of size $k-1$
by conjugation. We show that this action satisfies the conditions for a
Frobenius group given above.

Choose a prime $p$ dividing $k$. Since $p$ does not divide $n$, the subgroup
$N$ cannot contain a Sylow $p$-subgroup; but each of $H_1,\ldots,H_n$ contains
such a subgroup. By the conjugacy part of Sylow's theorem, $G$ acts
transitively on $\Omega$.

A non-identity element of $G$, acting by conjugation, fixes itself, and so
fixes the component containing it. We must show that it fixes no other
component. Count fixed point of elements of $G$. The identity fixes $n$;
non-identity elements of $N$ fix ${}\ge0$; and the remaining elements fix
${}\ge1$. So the sum of the fixed point numbers is at least $n+n(k-1)=nk$.
But, by the Orbit-counting Lemma, this sum is equal to $|G|=nk$, since
$G$ is transitive. So equality holds; non-identity elements of $N$ fix no
point, and the remaining non-identity elements fix exactly one point each.
So $G$ is a Frobenius group.
\end{proof}

Can we go further and identify an individual Frobenius group from its
commuting graph? The above analysis gives us the commuting graphs of the
Frobenius kernel and complement, so we would need to be able to recognise
these (albeit from rather restricted classes of groups). But we would also
need to identify the fixed-point-free action of $H$ on $N$, and it is not
clear how to do this from the graph. Given $H$ and $N$, we could simply
compute all fixed-point-free actions of $H$ on $N$, and conclude that
$G$ was the semidirect product given by one of these actions.

\section{Commuting graphs of $p$-groups of order $p^3$}


Let $G$ be a nonabelian $p$-group of order $p^3$. As it is nonabelian,
its center $Z(G)$ is $C_p$ (cyclic of order $p$). Furthermore, as
$G/Z(G)$ is abelian, being order $p^2$, and it cannot be cyclic for
otherwise $G$ would be abelian, $G/Z(G)=C_p\times C_p$.

It follows that the $G$-homomorphism $\phi: x\mapsto x^p$ has image
contained in $Z(G)=C_p$. Therefore, $\ker(\phi)$ is a subgroup of
order $p^2$ or $p^3$.\\

\noindent\textbf{Case 1: $\ker(\phi)$ is order $p^3$}~~ Then every
nontrivial element in $G$ has order $p$.

For any $x\notin Z$, we have $\angle{x,Z}=\{x^iz^j\mid 0\le i,j\le
p-1\}$ is a $G$-subgroup of order $p^2$ (which is always abelian).

Given the graph $X$, we can first identify $Z$ (the degree $n-1$
vertices). Take any vertex $x\notin Z$. Then $\{x\}\cup Z$ is a clique
which we can keep growing as follows: if $S$ is the current clique
pick any vertex $y$ not in $S$ such that $y$ is adjacent to all of $S$
and include it. Since $\angle{Z,x}$ is an abelian subgroup of $G$ and
is of order $p^2$, we will be able to grow the clique to precisely
$\angle{Z,x}$ and no further (because $G$ is not abelian).

We can repeat the above process of building a $p^2$-size clique
starting with a fresh vertex $x'$ each time to obtain $p+1$ cliques
corresponding to the subgroups of size $p^2$. These cliques all
intersect pairwise at $Z$ and are otherwise mutually disjoint.  There
will be no edges in the graph between a $z^ix^j$ and a $z^\ell y^m$ if
$x\ne y$ and $0< j,m\le p-1$.

Thus, we can recognize precisely the commuting graphs of such groups
of order $p^3$.\\

\noindent\textbf{Case 2: $\ker(\phi)$ is order $p^2$:}~~ In this case
there are elements in $G$ of order $p^2$ (not $p^3$ because $G$ is
nonabelian). Indeed, each element in $G\setminus \ker(\phi)$ is order
$p^2$. Clearly, the cyclic groups $\angle{x}$, $x\in G\setminus
\ker(\phi)$, all intersect precisely at the group's center $Z$ (like a
sunflower's center). That would account for
\[
\frac{(p^3-p^2)}{(p^2-p)} = p
\]
such cyclic groups of order $p^2$. These $p^3-p^2$ vertices in the
graph $X$ would form a ``sunflower'' of $p$ cliques, each of size
$p^2$ and all intersecting at the center $Z$.

The remaining $p^2$ elements are in $\ker(\phi)$ which will also form
a clique of size $p^2$. Thus, we again have a sunflower of $p+1$
cliques of size $p^2$ each that pairwise intersect precisely at
the center $Z$ and no other element is repeated.

This commuting graph structure is exactly as in the first case and can
be easily detected.

\section{Recognizing commuting graphs of extraspecial groups}\label{extrasp-sec}

Let $p$ be an odd prime. Let $G$ be an extraspecial group of order $p^{2n+1}$,
$n\ge 2$, and let $X=(V,E)$ be an undirected graph with $p^{2n+1}$.

Our goal is to design a polynomial-time algorithm that takes a simple
undirected graph $X=(V,E)$, with $p^{2n+1}$ vertices, as input and
determines if $X$ is the commuting graph of an extraspecial group $G$
of order $p^{2n+1}$. Moreover, the algorithm is required to find a
bijection $v\mapsto g_v$ labeling each vertex $v\in V$ by a unique
group element $g_v\in G$.

We recall that for an extraspecial group $G$ its centre $Z(G)$ is of
order $p$ and coincides with its derived subgroup $G'$ and the
Frattini subgroup $\Phi(G)$. Furthermore, it is known that there are
exactly two non-isomorphic extraspecial groups of order $p^{2n+1}$,
for each prime $p$. They are given by the following generator-relator
presentations:
\begin{itemize}
\item $G_1=\angle{z,x_i,y_i, 1\le i\le n}$ such that $[x_i,x_j]=1$,
  $[y_i,y_j]=1$ for all $i, j$, and $[x_i,y_j]=1$ for all $i\ne
  j$. $[x_i,y_i]=z$ for all $i$, and $x_i^p=y_i^p=z^p=1$ for all
  $i$. The center $Z(G_1)=\angle{z}$.

\item $G_2=\angle{z,x_i,y_i, 1\le i\le n}$ such that $[x_i,x_j]=1$,
  $[y_i,y_j]=1$ for all $i, j$, and $[x_i,y_j]=1$ for all $i\ne
  j$. And $[x_i,y_i]=z$ for all $i$, and $x_i^p=1=z^p$, $y_i^p=z$ for
  all $i$. The center $Z(G_2)=\angle{z}$.
\end{itemize}

It is clear from the above that $G_1$ and $G_2$ have isomorphic
commuting graphs. Thus, our aim is to design an algorithm that
identifies the vertices of the input graph $X$ with, say $G=G_1$, such
that all commuting pairs are realized by the edges of $X$.

First, notice that the vertices of $X$ corresponding to the centre
$Z=Z(G)$ of $G$ are easily detected as precisely the $p$ vertices in
$X$ of degree $p^{2n+1}-1$ each.

In either case ($G\in\{G_1,G_2\}$), we note that its centre $Z(G)$ can
be identified with the additive group of $\F_p$, and the quotient
$G/Z(G)$ with a $2n$-dimensional vector space $V(2n,p)$ over
$\F_p$. Moreover, the commutation map 
\[
(xZ(G),yZ(G))\mapsto [x,y]
\]
from $G/Z(G)\times G/Z(G()$ to $Z(G)$ defines a \emph{symplectic form}
on $V(2n,p)$, that is, a non-degenerate alternating bilinear form
$\sympl{,}$ on $V(2n,p)$ with values in $\F_p$. More precisely, we
recall the following.

Let $V(2n,p)$ be the $2n$-dimensional vector space over $\F_p$ which
is isomorphic to $G/Z(G)$, where $G$ is an extraspecial group of order
$p^{2n+1}$. Let $\sympl{,}$ be a \emph{symplectic form} on $V(2n,p)$.
I.e.,
\[
\sympl{,} : V(2n,p)\times V(2n,p)\to \F
\]
is a bilinear form that is \emph{antisymmetric} and $\sympl{u,u}=0$
for all vectors $u$.

We additionally know that $\sympl{,}$ is \emph{non-degenerate}: there
is no nonzero $u$ such that $\sympl{u,v}=0$ for all $v\in V(2n,p)$.
For a subset $S\subset V(2n,p)$ let
\[
S^\perp = \{u\in V(2n,p)\mid \sympl{u,v}=0 \text{ for all } v\in S\}
\]
denote the subspace of $V(2n,p)$ consisting of vectors
\emph{orthogonal} to each vector in $S$: $u$ is orthogonal to $v$ if
$\sympl{u,v}=0$.

The following is immediate.

\begin{lemma}\label{G-to-sympl}
From the extraspecial group $G$ given by its multiplication table as
input, we can construct its corresponding symplectic form $\sympl{,}$,
as an explicit bilinear map from $\F_p^{2n}\times \F_p^{2n}\to \F_p$,
in polynomial time.
\end{lemma}

\begin{definition}\label{orth-graph-def}
The \emph{orthogonality graph} of $V(2n,p)$ is an undirected simple
graph with $V(2n,p)$ as vertex set such that for each pair of distinct
vectors $u,v\in V(2n,p)$, $(u,v)$ is an undirected edge in the graph
iff $\sympl{u,v}=0$.
\end{definition}

For the actual input $X$, which is the purported commuting graph of
$G$, we now show that the orthogonality graph of the underlying
symplectic form can be efficiently computed.

\begin{lemma}\label{commute-to-ortho}
  Given as input a candidate commuting graph $X=(V,E)$, of the
  extraspecial group $G$ of order $p^{2n+1}$, we can compute in
  polynomial time an undirected graph $X_o=(V_o,E_o)$ with
  $|V_o|=p^{2n}$ such that $X_o$ is the orthogonality graph of the
  symplectic form $\sympl{,}$ if and only if $X$ is the commuting
  graph of $G$.
\end{lemma}

\begin{proof}
If $X$ is the commuting graph of $G$ then the subset of all
\emph{dominating vertices} of $X$ (i.e., vertices of degree $|V|-1$ in
$X$) corresponds precisely to the elements of the centre
$Z(G)$. Hence, there are exactly $p$ of them in $X$. Let $Z$ denote
this subset of vertices. If that number is different from $p$ then we
can reject $X$ as being the commuting graph of an order $p^{2n+1}$
extraspecial group.

Next, recall that the \emph{closed neighborhood} $\bar{N}(v)$ of a
vertex $\in V$ is defined as
\[
\bar{N}(v) = \{u\in V\mid (u,v)\in E\} \cup \{v\}.
\]
Vertices $v_1,v_2,\in V$ are called \emph{closed twins} if
$\bar{N}(v_1)=\bar{N}(v_2)$.  We can identify all the subgroups of $G$
of order $p^2$ by defining the following equivalence relation on
$V\setminus Z$. This clearly defines an equivalence relation on
vertices in $V\setminus Z$.

For each order $p^2$ subgroup $H\le G$ such that $Z(G)\le H\le G$, it
is easy to see that for all $h,h'\in H\setminus Z(G)$ their
centralizers in $G$ coincide: $C_G(h)=C_G(h')$. Hence, if $X$ is the
commuting graph of $G$, then $\bar{N}(u)=\bar{N}(v)$ for all $u,v\in
U\setminus Z$, where $U$ corresponds to $H$. On other hand, if $u\in
H\setminus Z$ and $v\notin H$ then we can use the symplectic form
$\sympl{,}$ to see that for any two linearly independent vectors
$v_1,v_2\in V(2n,p)$ we can find a vector orthogonal to exactly one of
them (by expressing $v_1,v_2$ using the basis $\{e_i,f_i, 1\le i\le
n\}$). Thus, the equivalence classes defined by the closed-twins
equivalence relation identifies all $p^2-p$ size vertex subsets of
$V\setminus Z$ that corresponds to $H\setminus Z(G)$ for each subgroup
$H$ of order $p^2$ such that $Z(G)\le H\le G$. The number of such
equivalence classes is
\[
\frac{p^{2n+1}-p}{p^2-p} = 1 + p + p^2 +\cdot + p^{2n-1}.
\]

To obtain the corresponding orthogonality graph $X_o=(V_o,E_o)$ we
apply the following three steps to $X$.

\begin{itemize}
\item[(a)] Collapse each equivalence class in $V\setminus Z$ into a
  single vertex. This gives rise to a set $V_c$ of
  $\sum_{i=0}^{2n-1}p^i$ many vertices. The edges between vertices
  in $V_c$ are naturally inherited from $X$.
\item[(b)] Include a new vertex $v_0$ corresponding to the $0$ element
  of $V(2n,p)$; this corresponds to $p$ dominant vertices of $X$.
\item[(c)] The process of collapsing the closed-twin equivalence
  classes identifies vectors $v\in V(2n,p)$ with all the $p-1$ nonzero
  scalar multiples $\alpha v, \alpha\in\F^*_p$. We restore the $p-1$
  copies by replacing each $v\in V_c$ by $p-1$ copies. The edges
  between these vertices are naturally inherited.
\end{itemize}


By construction, $X_o$ is the orthogonality graph of the symplectic
form on $V(2n,p)$ defined by the commutation map of $G$ if and only if
$X$ is the commuting graph of $G$.
\end{proof}

\subsection{Recognizing the orthogonality graph of the symplectic form}

We now show that the orthogonality graph of a non-degenerate
symplectic form on $V(2n,p)$ can be recognized in polynomial time.

\begin{theorem}\label{ortho-recog-thm}
Given a simple undirected graph $X=(V,E)$ on a vertex set $V$ of size
$p^{2n}$ vertices, in time polynomial in the size of $X$ we can
recognize if $X$ is the orthogonality graph of some symplectic form on
$V(2n,p)$ and, if so, in polynomial time we can also compute a
bijection from $V$ to $V(2n,p)$ and determine a symplectic form
$\sympl{,}$ that is consistent with $X$.
\end{theorem}

\begin{proof}
  Let $X=(V,E)$ be a graph with $p^{2n}$ vertices, $p$ prime. Our
  algorithm is based on an inductive argument. We know that $V(2n,p)$
  has a \emph{symplectic basis} of $2n$ vectors of the form $e_i, f_i,
  1\le i\le n$ such that
\begin{itemize}
\item $\sympl{e_i,f_i}=1$ for all $i$.
\item $\sympl{e_i,f_j}=0$ for all $i\ne j$.
\item $\sympl{e_i,e_j}=0$ for all $i,j$.
\item  $\sympl{f_i,f_j}=0$ for all $i,j$.
\end{itemize}

Moreover, we can construct such a basis for $V(2n,p)$ by the following
greedy process. Pick $e_1\ne 0$ in $V(2n,p)$ arbitrarily. Then pick
any vector $f_1\ne 0$, suitably scaled, such that
$\sympl{e_1,f_1}=1$. Notice that $f_1$ must exist as the symplectic
form is non-degenerate. Then notice that the subspace
$\{e_1,f_1\}^\perp$ is a $2n-2$ dimensional symplectic space $V'$
(w.r.t. the same symplectic form). We can continue with the basis
construction by induction applied to $V'$.

The problem we will solve is to efficiently simulate the above
construction process given only the purported commuting graph
$X=(V,E)$ as input. First of all, the zero vector $0$ is the only
vertex in $X$ adjacent to all others and is easily identified. Let
$e_1$ be any other vertex. We will choose $f_1$ as any vertex not
adjacent to $e_1$. The scaling factor does not matter since a vector
$v$ and $\alpha v, \alpha\in\F^*_p$ have identical neighborhoods in
the orthogonality graph of the symplectic space $V(2n,p)$. Let
\[
V'=\{v\in V\mid (v,e_1)\in E \text{ and } (v,f_1)\in E\}.
\]
That is, $V'$ is the common neighborhood of $e_1$ and $f_1$. Let $X'$
be the subgraph of $X$ induced by the vertex subset $V'$.  We have the
following easy observation.
\begin{claim}  
If $X$ is the orthogonality graph of a $2n$-dimensional symplectic
space over $\F_p$ then the graph $X'$ is the orthogonality graph of a
symplectic space of dimension $2n-2$ over $\F_p$.
\end{claim}
Inductively, therefore, we assume that we have checked that $X'$ is
indeed the orthogonality graph of a $2n-2$-dimensional symplectic
space over $\F_p$ and we have a labeling of the vertices of $V'$ by
linear combinations $\sum_{i=2}^n(\alpha_i e_i + \beta_i f_i),
\alpha_i,\beta_i\in \F_p$ that is consistent with the orthogonality
relation of a symplectic form $\sympl{,}$.

The remaining task for the algorithm is to find a consistent labeling
of the vertices in $V\setminus V'$.

\begin{claim}
A vertex $v\in V\setminus V'$ can be labeled by a nonzero vector
$\alpha e_1 + \beta f_1$ in $V(2n,p)$, $\alpha,\beta\in\F_p$, if and
only if $(v,u)\in E$ for all $u\in V'$.
\end{claim}

\claimproof{Consider the orthogonality graph of $V(2n,p)$. Let
  $V(2n-2,p)$ denote the subspace spanned by $e_i,f_i, 2\le i\le n$.
  Clearly, every vector of the form $\alpha e_1 + \beta f_1,
  \alpha,\beta\in\F_p$ is orthogonal to each vector in
  $V(2n-2,p)$. Conversely, consider a vector $\alpha e_1 + \beta f_1
  +v\in V(2n,p)$, where $v\in V(2n-2,p)$ is nonzero. Since $V(2n-2,p)$
  is non-degenerate, there is a $u\in V(2n-2,p)$ such that
  $\sympl{v,u}\ne 0$ which implies $\sympl{\alpha e_1 +\beta f_1 +
    v,u}=\sympl{v,u}\ne 0$.} 

Thus, we will find precisely $p^2-1$ many such vertices in $V$ that are
adjacent to all of $V'$, of which we have already labeled two vertices
as $e_1$ and $f_1$.

The next claim is also clear from the construction of the $e_i, f_i$
basis.

\begin{claim}
  If $X=(V,E)$ is the orthogonality graph of the symplectic space
  $V(2n,p)$ and $V'$ corresponds to the subspace $V(2n-2,p)$ spanned
  by $\{e_i,f_i\mid 2\le i\le n\}$ then the subset
  \[
  V[e_1]=\{v\in V\setminus V'\mid (v,e_1)\in E \text{ and } (v,f_1)\notin E\}
  \]
  consists precisely of those vertices in $V\setminus V'$ that
  correspond to the subset of vectors $\{\alpha e_1 +V'\mid \alpha\ne
  0\}$ of $V(2n,p)$.
  
  Furthermore, the vertices that correspond to $\alpha e_1, \alpha\ne
  0$ are precisely those vertices in $V[e_1]$ whose neighborhood in
  $X$ is identical to the neighborhood of the vertex labeled $e_1$.
\end{claim}

From the above claim it follows that we can identify the vertex subset
$V[e_1]$, corresponding to $\alpha e_1 +V', \alpha\ne 0$. Similarly,
we can identify $V[f_1]$, corresponding to $\beta f_1 +V', \beta\ne
0$.\\

\noindent\textbf{Labeling vertices in $V[e_1]$ and $V[f_1]$}\\

Consider the vertex subset $V[e_1]$. For each $u'\in V(2n-2,p)$ and
$u\in V(2n-2,p)$, notice that
\[
\sympl{u',u}=0 \text{ if and only if } \sympl{u',\alpha e_1 +u} = 0,
\alpha\in \F^*_p,
\]
because $\sympl{u',e_1}=0$ for all $u'\in V(2n-2,p)$.

Let $N(u,V')$ denote the neighborhood of $u$ in $V'$. The above
statement is equivalent to saying that for each vertex $u\in V'$ there
are exactly $p-1$ vertices $u''\in V[e_1]$ such that
\[
N(u'',V')=N(u,V').
\]
We can label these $p-1$ vertices arbitrarily as $\alpha e_1 +u$, for
$\alpha\in \F_p^*$. Thus, in polynomial time, we can obtain the
correct labeling of all vertices in $V[e_1]$ by the vectors in $\alpha
e_1 + V(2n-2,p), \alpha\ne 0$.

Similarly, we can obtain the correct labeling of $V[f_1]$ by the
vectors in $\beta f_1 + V(2n-2,p)$.\\

\noindent\textbf{Labeling vertices in $V[e_1+\beta f_1], \beta\ne 0$}\\

For $\beta\ne \beta'\in \F_p$, we have
\[
\sympl{e_1+\beta f_1,e_1+\beta' f_1}= \beta'-\beta\ne 0.
\]
Hence, the vertices to be labeled $e_1+\beta f_1$ and $e_1+\beta' f_1$
are not adjacent in $X$. Since $\sympl{e_1+\beta
  f_1,v}=0=\sympl{e_1+\beta' f_1,v}$ for all $v\in V'$, these two
vertices are adjacent to each vertex in $V'$.

Notice that these two vertices could have been picked instead of $e_1$
and $f_1$ in the first place and we would have still obtained the same
subset $V'$ as the subspace $V(2n-2,p)$. Thus, we can repeat the
previous argument (for labeling vertices of $V[e_1]$) to label each of
the vertex subsets $V[e_1+\beta f_1]$ for $\beta\in \F_p^*$.\\

Putting it together, clearly if $X$ is the orthogonality graph of some
non-degenerate symplectic form on $V(2n,p)$ the above algorithm
verifies that by constructing a (possibly different) symplectic form
on $V(2n,p)$, consistent with the orthogonality graph.\\

\noindent\textbf{Running time analysis}~~Let $T(n)$ denote the
running time for constructing a symplectic form $\sympl{,}$ for
$V(2n,p)$ consistent with $X$ as the orthogonality graph. The
inductive construction and the rest of the computation implies the
recurrence $T(n) = O(|V|^2) +T(n-1)$, which gives an overall cubic
bound $T(n)=O(|V|^3)=O(p^{6n+3})$ on the running time.

Putting it together, we have a polynomial-time algorithm that checks
if $X$ is the orthogonality graph for the symplectic space by finding
out a labeling of vertices by the vectors along with a consistent
symplectic form.
\end{proof}

Putting everything together, we have the following.

\begin{theorem}
  Given $X=(V,E)$ with $p^{2n+1}$ vertices we can determine in
  polynomial time if it is the commuting graph of an extraspecial
  group of order $p^{2n+1}$ and, if so, label vertices by unique
  group elements satisfying the commuting relation.
\end{theorem}

\begin{proof}
  By Lemma~\ref{commute-to-ortho} we can obtain the graph
  $X_o=(V_o,E_o)$ from $X$ in polynomial time. By
  Theorem~\ref{ortho-recog-thm} we can check if $X_o$ is the
  orthogonality graph of a symplectic form $\sympl{,}$ on $V(2n,p)$
  and also find the symplectic form.

  From the proof of Lemma~\ref{commute-to-ortho}, we have the
  following observations:
\begin{itemize}  
\item The vertex of degree $|V_o|-1$ in $X_o$ is the $0$ element of
  $V(2n,p)$, and corresponds to $Z$ in $G$.
\item For each vertex $x\in V_o$ there is a closed-twins equivalence
  class of size $p-1$ containing $x$. There is a corresponding subset
  of vertices $H_x\setminus Z$ in $X$ of size $p^2-p$, where $H_x$, in
  turn, corresponds to an order-$p^2$ subgroup of $G$ that contains
  $Z$. We have a labeling of the vertex $x$ by a linear combination
  $\sum_{i=1}^n( \alpha_i e_i + \beta_i f_i)$, and the remaining $p-2$
  vertices in the equivalence class are labeled by nonzero scalar
  multiples of this linear combination.
\end{itemize}  

  Now, using the description of the extraspecial group $G_1$ we can
  label the vertices of the clique $H_x\setminus Z$ with the group
  elements
\[
  z^j\cdot \left(\prod_{i=1}^n x_i^{\alpha_i}.y_i^{\beta_i}\right)^k,
  0\le j\le p-1, 1\le k \le p-1.
\]  
  
 This will give us the labeling of $X$ with the elements of the
 extraspecial group $G_1$ of order $p^{2n+1}$, consistent with the
 commuting graph $X$. This completes the proof.
\end{proof}

\section{A quasipolynomial time algorithm in the general case}

We now describe a $2^{O(\log^3n)}$ time algorithm for checking if a
given $n$-vertex graph is the commuting graph of some $n$-element
group and, if so, labeling the vertices of the graph by the group
elements consistent with all the commuting pairs.

This algorithm is based on a result of McIver and Neumann \cite{MN87}
that bounds the number of $n$ element groups by $2^{O(\log^3n)}$. The
theorem is sharpened by subsequent work of Babai et al \cite{Bab} on
short presentations for finite groups. They show that groups of order
$n$ (that do not have a specific finite simple group type, called the
Ree groups, as section) have short generator-relator presentations
$\angle{Y|R}$ of size $O(\log^3n)$. The crux of their proof is that
all simple groups of $n$ elements (except the Ree groups) have
generator-relator presentations of size $O(\log^2 n)$. Combined with
the fact that $n$ element groups have composition series of $\log n$
length gives the $O(\log^3n)$ bound.

\begin{theorem}
There is a $2^{O(\log^3n)}$ time algorithm that recognizes if a given
$n$-vertex graph is the commuting graph of a group of order $n$ that
does not have the Ree groups as section.
\end{theorem}

\begin{proof}
Given such a generator-relator presentation $\angle{Y|R}$ we first
find the multiplication table of the group in polynomial time. More
precisely, let $G$ be a group of order $n$ with composition series:
\[
1=N_0\rhd N_1\rhd \cdots \rhd N_r=G,
\]
where $r\le \log n$. For each $i\ge 1$, the quotient group
$N_i/N_{i-1}$ is simple of some order $n_i$, where $\prod_{i=1}^r n_i
=n$.  By the above mentioned theorem of \cite{Bab}, each $N_i/N_{i-1}$
has an $O(\log^2 n_i)$ size generator-relator presentation
$\angle{Y_i|R_i}$. Inductively assume that we have computed the group
multiplication table for $N_{i-1}$. The generating set $Y_i$ for
$N_I/N_{i-1}$ is a collection of cosets $yN_{i-1}$. Combined with the
multiplication table for $N_{i-1}$, and using the relations in $R_i$,
we can compute the multiplication table for $N_i$ in polynomial time.

Continuing thus, we will have the multiplication table for the entire
group $G$ from which we can find its commuting graph $\Gamma(G)$.
Now, we can test if $\Gamma(G)$ is isomorphic to the input graph $X$
using Babai's $2^{O(\log^3n)}$ time algorithm.

Since each presentation $\angle{Y|R}$ is of size $O(\log^3n)$, we can
go through all of them in time $O(\log^3n)$, finding the commuting
graph $\Gamma(G)$ for the corresponding group $G$ and then running
Babai's isomorphism test to check if $\Gamma(G)\simeq X$.

Thus, the overall computation takes $2^{O(\log^3n)}$ time.
\end{proof}

\section{Isoclinism of groups and commuting graphs}


As extraspecial groups are a special case of nilpotent groups of class
$2$, a natural question is whether the algorithm of
Section~\ref{extrasp-sec} can be extended to efficiently recognize the
commuting graphs of nilpotent groups of class $2$.

The property of extraspecial groups that we exploited in the algorithm
is that extraspecial groups are \emph{isoclinic}. We briefly recall
the definition and its connection to commuting graphs \cite{GGsurvey}:

Clearly, two isomorphic groups have the same commuting graph (meaning
isomorphic commuting graphs). We can define an equivalence relation
among finite groups of order $n$: two groups are equivalent if they
have the same commuting graph.
 
 The \emph{commutator map} of $G$ is the map
\[
\kappa~: (Za,Zb)\mapsto aba^{-1}b^{-1}
\]
from the product $G/Z(G)\times G/Z(G)$ to the commutator subgroup
$G'$.

Two groups $G_1$ and $G_2$ are \emph{isoclinic} if $G_1/Z_1$ and
$G_2/Z_2$ are isomorphic, and their derived subgroups $G'_1$ and
$G'_2$ are isomorphic via isomorphisms that commute with the $\kappa$
map.

Suppose $G_1$ and $G_2$ are isoclinic groups such that their centers
$Z_1$ and $Z_2$ are of the same order. Then, first of all, the
commuting graphs of $G_1/Z_1$ and $G_2/Z_2$ are isomorphic because the
groups are isomorphic. The commuting graph of $G_i$ can be obtained
from the commuting graph of $G_i/Z_i$ by correctly blowing up each
coset vertex to a coset of vertices (and including the edges as
required: two vertex cosets are either fully connected with each other
or not connected at all). The isoclinism property ensures that the
commuting graphs of $G_1$ and $G_2$ remain isomorphic.

What about the converse? That is, if two groups have isomorphic commuting
graphs, must they be isoclinic? This holds for various classes of groups,
such as abelian groups, nonabelian simple groups, and extraspecial groups
(as we have seen).

For extraspecial groups the converse property was exploited in
obtaining the efficient recognition algorithm for their commuting
graphs. 

However, it is not true in general; we recycle an example taken from
\cite{CK} to show this. Let $G$ be the group of order~$64$ which is
\texttt{SmallGroup(64,182)} in the SmallGroups library in \textsf{GAP}
\cite{gap}. The Schur multiplier of $G$ has order $2$, so a Schur cover (a
group $H$ of maximal order subject to having a subgroup $Z\le H'\cap Z(H)$
such that $H/Z\cong G$) has order $128$. Moreover, the Bogomolov multiplier
of $G$ is equal to the Schur multiplier, which implies that the Schur covers
$H$ are \emph{commutation-preserving}: that is, two elements $a,b\in H$
commute if and only if their projections $Za,Zb\in G$ commute. This implies
that the commuting graph of a Schur cover is obtained from the commuting graph
of $G$ by replacing each vertex with a clique of size $2$, with all edges
between cliques corresponding to adjacent vertices. This procedure also
describes the commuting graph of $G\times C_2$. On the other hand, it is
easy to verify computationally that the derived groups of $H$ and $G\times C_2$
are not isomorphic, so these groups cannot be isoclinic.

We note that the group $G$ has nilpotency class $3$, as do all of its Schur
covers (these are \texttt{SmallGroup(128,$i$)} for
$i\in\{789, 790, 791, 815, 816, 817\}$ in the \textsf{GAP} library. So the
following question is still open:

\paragraph{Conjecture} Is it true that a nilpotent group of class~$2$ is
determined up to isoclinism by its commuting graph?



\newcommand{\etalchar}[1]{$^{#1}$}

\end{document}